\documentclass[twoside, 10pt]{article}
\usepackage{amssymb, amsmath, mathrsfs, amsthm}
\usepackage{graphicx}
\usepackage{color}
\usepackage{float, caption, subcaption}

\newcommand{\bd}{\begin{description}}
\newcommand{\ed}{\end{description}}
\newcommand{\bi}{\begin{itemize}}
\newcommand{\ei}{\end{itemize}}
\newcommand{\be}{\begin{enumerate}}
\newcommand{\ee}{\end{enumerate}}
\newcommand{\beq}{\begin{equation}}
\newcommand{\eeq}{\end{equation}}
\newcommand{\beqs}{\begin{eqnarray*}}
\newcommand{\eeqs}{\end{eqnarray*}}

\newcommand{\ceil}[1]{\left\lceil #1 \right\rceil}

\definecolor{DarkGreen}{rgb}{0.2, 0.6, 0.3}

\newcommand{\pp}{\prime \prime}

\newtheorem{theorem}{Theorem}
\newtheorem{conjecture}{Conjecture}

\newtheorem{lemma}{Lemma}

\newtheorem{case}{Case}
\newtheorem{subcase}{Subcase}[case]
\newtheorem{claim}{Claim}
\newtheorem{fact}{Fact}

\begin{document}
\title{\textbf{Gallai-Ramsey numbers for books\footnote{Supported by the National Science Foundation of China
		(Nos. 11601254, 11551001, 11161037, and 11461054) and the Science
		Found of Qinghai Province (Nos. 2016-ZJ-948Q, and 2014-ZJ-907).}}}
\author{Jinyu Zou\footnote{School of Computer, Qinghai Normal University, Xining, Qinghai 810008, China} \footnote{Department of Basic Research, Qinghai University, Xining, Qinghai 810008, China}, Yaping Mao\footnote{School of Mathematics and Statistis, Qinghai Normal University, Xining, Qinghai 810008, China}, Colton Magnant\footnote{Department of Mathematical Sciences, Georgia Southern University, Statesboro, GA, 30460, USA},\\ Zhao Wang\footnote{School of Mathematical Sciences, Beijing Normal University, Beijing 100875, China. {\tt wangzhao380@yahoo.com}}, 
Chengfu Ye\footnotemark[2]}

\maketitle

\begin{abstract}
Given a graph $G$ and a positive integer $k$, the \emph{Gallai-Ramsey number} is defined to be the minimum number of vertices $n$ such that any
$k$-edge coloring of $K_n$ contains either a rainbow (all different colored) triangle
or a monochromatic copy of $G$. In this paper, we obtain general upper and lower bounds
on the Gallai-Ramsey numbers for books $B_{m} = K_{2} + \overline{K_{m}}$ and prove sharp results for $m \leq 5$.
\end{abstract}

\section{Introduction}

In this work, we consider only edge-colorings of graphs. A coloring of a graph $G$ is called
\emph{rainbow} if no two edges in $G$ have the same color.

Edge colorings of complete graphs that contain no rainbow triangle have very interesting and somewhat surprising structure. In 1967, Gallai \cite{Gallai} examined this structure
under the guise of transitive orientations (a translation of this paper is available in \cite{RamirezReed}).
In honor of Gallai's result (Theorem~\ref{Thm:G-Part}), such colorings are are called \emph{Gallai-colorings}. The result was restated in \cite{GyarfasSimonyi} in the terminology of graphs and can also be traced back
to \cite{CameronEdmonds}. For the following statement, a trivial partition is a partition into only one part.

\begin{theorem}{\upshape \cite{CameronEdmonds, Gallai, GyarfasSimonyi}}\label{Thm:G-Part}
In any coloring of a complete graph containing no rainbow
triangle, there exists a nontrivial partition of the vertices (called a Gallai-partition) such
that there are at most two colors on the edges between the parts and only one color on
the edges between each pair of parts.
\end{theorem}

For ease of notation, we often refer to a Gallai-colored graph as \emph{G-colored} and a Gallai-partition as a \emph{G-partition}. The induced subgraph of a G-colored complete graph constructed by selecting a single vertex from each part of a G-partition is called a \emph{reduced graph}. By Theorem~\ref{Thm:G-Part}, the reduced graph is a colored complete graph using at most $2$ colors.

Given two graphs $G$ and $H$, let $R(G, H)$ denote the $2$-color Ramsey number for finding a monochromatic $G$ or $H$, that is, the minimum number of vertices $n$ needed so that every red-blue coloring of $K_{n}$ contains either a red copy of $G$ or a blue copy of $H$. Although the reduced graph of a Gallai partition uses only two colors, the original Gallai-colored complete graph could certainly use more colors. With this in mind, we consider the following generalization of the Ramsey numbers. Given two graphs $G$ and $H$, the \emph{general $k$-colored Gallai-Ramsey number} $gr_k(G:H)$ is defined to be the minimum integer $m$ such that every $k$-coloring of the complete graph on $m$ vertices contains either a rainbow copy of $G$ or a monochromatic copy of $H$. With the additional restriction of forbidding the rainbow copy of $G$, it is clear that $gr_k(G:H)\leq R_k(H)$ for any $G$.

Although the perspective of forbidding a rainbow triangle and looking for monochromatic subgraphs can be traced back to \cite{GyarfasSimonyi} and beyond, the study of Gallai-Ramsey numbers in their current form comes from \cite{FGJM10}, where the authors used Theorem~\ref{Thm:G-Part} to produce sharp Gallai-Ramsey numbers for several small graphs along with some general bounds for paths and cycles. The subject has also been expanded through several other publications including but not limited to \cite{FM11, GSSS10}. In particular, in \cite{GSSS10}, the following general behavior of Gallai-Ramsey numbers was established.

\begin{theorem}[\cite{GSSS10}]\label{Thm:Dichotomy}
Let $H$ be a fixed graph with no isolated vertices. If $H$ is bipartite and not a star, then $gr_{k}(K_{3} : H)$ is linear in $k$. If $H$ is not bipartite, then $gr_{k}(K_{3} : H)$ is exponential in $k$.
\end{theorem}

There is a survey containing these and related results in \cite{FMO10} with a dynamically updated version available in \cite{FMO14}.

More specifically (and more recently), Fox, Grinshpun, and Pach proposed the general study of Gallai-Ramsey numbers for complete graphs with the following conjecture.

\begin{conjecture}[\cite{FGP15}]\label{Conj:Fox}
For $k\ge 1$ and $p \geq 3$,
$$
gr_{k}(K_{3} : K_{p}) = \begin{cases}
(R(K_{p}, K_{p}) - 1)^{k/2} + 1 & \text{ if $k$ is even,}\\
(p - 1)(R(K_{p}, K_{p}) - 1)^{(k - 1)/2} + 1 & \text{ if $k$ is odd.}
\end{cases}
$$
\end{conjecture}

The case where $p = 3$ was actually verified in 1983 by Chung and Graham \cite{CG83} and then again in \cite{AI08} in a different context, long before the conjecture. A simplified proof was given by Gy\'arf\'as et al.~\cite{GSSS10}.

\begin{theorem}[\cite{AI08, CG83,GSSS10}]
For $k \geq 1$,
$$
gr_{k}(K_{3} : K_{3}) = \begin{cases}
5^{k/2} + 1 & \text{if $k$ is even,}\\
2\cdot 5^{(k-1)/2} + 1 & \text{if $k$ is odd.}
\end{cases}
$$
\end{theorem}

The next open case, when the desired monochromatic graph is $K_{4}$, was recently verified.

\begin{theorem}[\cite{GRK4}]
For $k\ge 1$,
$$
gr_{k}(K_{3} : K_{4}) = \begin{cases} 17^{k/2} + 1 & \text{ if } k \text{ is even,}\\
3\cdot 17^{(k - 1)/2} + 1 & \text{ if } k \text{ is odd.}
\end{cases}
$$
\end{theorem}

Note that these numbers depend heavily upon the $2$-color Ramsey number so the next case of this conjecture will not yield a sharp result until $R(K_{5}, K_{5})$ is discovered. This makes the study of Gallai-Ramsey numbers for non-complete graphs a more natural and interesting endeavor.

The \emph{book} graph with $m$ \emph{pages} is denoted by $B_m$, where $B_{m} = K_{2} + \overline{K_{m}}$. See Figure~\ref{Fig:B4} for a drawing of $B_{4}$. Call the central edge (the $K_{2}$) the \emph{spine} of the book. Note that $B_{1} = K_{3}$ and $B_{2} = K_{4} \setminus \{e\}$ where $e$ is an edge of the $K_{4}$. In this work, we prove bounds on the Gallai-Ramsey number of all books, with sharp results for several small books. Since books are not bipartite, in light of Theorem~\ref{Thm:Dichotomy}, it should come as no surprise that all of our results are exponential as a function of the number of colors $k$.

\begin{figure}[H]
\begin{center}
\includegraphics{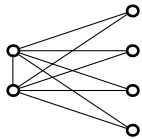} 
\end{center}
\caption{The book graph $B_{4}$ \label{Fig:B4}}
\end{figure}

The outline of the paper is as follows. In the short Section~\ref{Sec:Prelim}, we recall some known results that will be used later. Section~\ref{Sec:Low} contains a general lower bound while Section~\ref{Sec:Up} contains a general upper bound. We provide the sharp Gallai-Ramsey number for several small books in Section~\ref{Sec:Small} and finally, in Section~\ref{Sec:Con}, we state a conjecture generalizing the sharp results of Section~\ref{Sec:Small}.

\section{Preliminaries}\label{Sec:Prelim}

In this section, we briefly recall some tools that will be helpful in the remainder of this work. The first result was actually one of the first results on monochromatic subgraphs of rainbow triangle free colorings.

\begin{theorem}[\cite{GyarfasSimonyi}]\label{Thm:Stars}
In any G-coloring of a complete graph, there is a vertex with at least $\frac{2n}{5}$ incident edges in a single color.
\end{theorem}

We also recall some known $2$-color Ramsey numbers of small books.

\begin{theorem}[\cite{CH72, RS78}]\label{Thm:RamseyBm}
\beqs
R(B_{2}, B_{2}) & = & 10,\\
R(B_{3}, B_{3}) & = & 14,\\
R(B_{4}, B_{4}) & = & 18, \text{ and}\\
R(B_{5}, B_{5}) & = & 21.
\eeqs
\end{theorem}


Note that since $B_{1} = K_{3}$ and $B_{2} = K_{4} - e$, the following is known.

\begin{fact}[\cite{ChH2}]\label{Fact:B1B2} 
$$
R(B_{1}, B_{2}) = 7.
$$
\end{fact}


\section{General Lower Bound}\label{Sec:Low}

In this section, we prove a lower bound on the Gallai-Ramsey number for books by a straightforward inductive construction.

\begin{theorem}\label{Thm:GR-Bm-LowBd}
If $B_{m}$ is the book with $m$ pages, $B_{m} = K_{2} + \overline{K_{m}}$, then for $k \geq 2$,
$$
gr_{k}(K_{3} : B_{m}) \geq \begin{cases}
(R(B_{m}, B_{m}) - 1) \cdot 5^{(k - 2)/2} + 1 & \text{if $k$ is even,}\\
2 \cdot (R(B_{m}, B_{m}) - 1) \cdot 5^{(k - 3)/2} + 1 & \text{if $k$ is odd.}
\end{cases}
$$
\end{theorem}

\begin{proof}
We prove this result by inductively constructing a coloring of $K_{n}$ where
$$
n = \begin{cases}
(R(B_{m}, B_{m}) - 1) \cdot 5^{(k - 2)/2} & \text{if $k$ is even,}\\
2 \cdot (R(B_{m}, B_{m}) - 1) \cdot 5^{(k - 3)/2} & \text{if $k$ is odd,}
\end{cases}
$$
which contains no rainbow triangle and no monochromatic copy of $B_{m}$. For the base of this induction, let $G_{2}$ be a $2$-colored complete graph on $R(B_{m}, B_{m}) - 1$ vertices containing no monochromatic copy of $B_{m}$. Without loss of generality, suppose this coloring uses colors $1$ and $2$.

Suppose we have constructed a coloring of $G_{2i}$ where $i$ is a positive integer and $2i < k$, using the $2i$ colors $1, 2, \dots, 2i$ and having order $n_{2i} = (R(B_{m}, B_{m}) - 1) \cdot 5^{i - 1}$ such that $G_{2i}$ contains no rainbow triangle and no monochromatic copy of $B_{m}$.

If $k = 2i + 1$, we construct $G_{2i + 1} = G_{k}$ by making two copies of $G_{2i}$ and inserting all edges between the copies in color $k$. Then $G_{k}$ certainly contains no rainbow triangle, no monochromatic copy of $B_{m}$, and has order $n = 2 \cdot (R(B_{m}, B_{m}) - 1) \cdot 5^{(k - 3)/2}$.

Otherwise suppose $k \geq 2i + 2$. We construct $G_{2i + 2}$ by making five copies of $G_{2i}$ and inserting edges of colors $2i + 1$ and $2i + 2$ between the copies to form a blow-up of the unique $2$-colored $K_{5}$ which contains no monochromatic triangle, see Figure~\ref{Fig:K5}. This coloring clearly contains no rainbow triangle and, since there is no monochromatic triangle in either of the two new colors, there can be no monochromatic copy of $B_{m}$ in $G_{2i + 2}$, completing the construction.
\end{proof}

\begin{figure}[H]
\begin{center}
\includegraphics{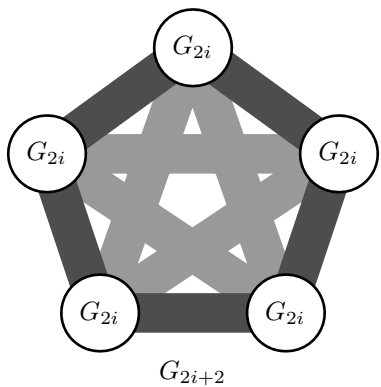}
\end{center}
\caption{An example of this construction \label{Fig:K5}}
\end{figure}

Note that if more than five copies of $G_{2i}$ were used in this construction, there must exist a monochromatic triangle on the edges between the copies, which would produce a monochromatic copy of $B_{m}$ if the copies are large enough.

\section{General Upper Bound}\label{Sec:Up}

Let $R_{m} = R(B_{m} : B_{m})$ and define
$$
R_{m}' = \sum_{i = 1}^{m - 1} [R_{\lceil m/i \rceil} - 1].
$$
This quantity provides a bound on a type of restricted Ramsey number as seen in the following lemma.

\begin{lemma}\label{Lemma:r=0}
For $m \geq 2$, the largest number of vertices in a G-coloring of a complete graph with no monochromatic $B_{m}$ in which all parts of the G-partition have order at most $m - 1$ is at most $R_{m}'$.
\end{lemma}

\begin{proof}
Suppose, for a contradiction, that $G$ is a $k$-coloring of $K_{n}$ containing no rainbow triangle and no monochromatic copy of $B_{m}$ where
$$
n = \sum_{i = 1}^{m - 1} [R_{\lceil m/i \rceil} - 1] + 1,
$$
and suppose all parts of a G-partition of $G$ have order at most $m - 1$. There are certainly at most $R_{m} - 1$ parts in this G-partition since otherwise the reduced graph would contain a monochromatic copy of $B_{m}$. Similarly, there are at most $R_{\lceil m/2 \rceil} - 1$ parts of order at least $2$ in the G-partition since otherwise the reduced graph of this subset of parts would contain a monochromatic copy of $B_{\lceil m/2 \rceil}$, which implies the existence of a monochromatic copy of $B_{m}$ in $G$.

More generally, for each integer $i$ with $1 \leq i \leq m - 1$, there are at most $R_{\lceil m/i \rceil} - 1$ parts of order at least $i$ in the G-partition since otherwise the reduced graph of this subset of parts would contain a monochromatic copy of $B_{\lceil m/i \rceil}$, which creates a monochromatic copy of $B_{m}$ in $G$. Summing over the values of $i$, we see that
$$
|G| = n \leq \sum_{i = 1}^{m - 1} [R_{\lceil m/i \rceil} - 1] < n,
$$
a contradiction, completing the proof of Lemma~\ref{Lemma:r=0}.
\end{proof}

Note that for $m$ large, since $R_{m} \sim (4 + o(1))m$ (see \cite{RS78}), we get $R_{m}' \sim (4 + o(1))m\ln [(4 + o(1))m]$. In particular, for some small values of $m$, we compute
\beqs
R_{2}' & = & 9,\\
R_{3}' & = & 22,\\
R_{4}' & = & 35, \text{ and}\\
R_{5}' & = & 51.
\eeqs

Call a color \emph{$m$-admissible} if it induces a subgraph with maximum degree at least $m$, and \emph{$m$-inadmissible} otherwise.

\begin{lemma}\label{Lemma:1admissible}
Given integers $m \geq 2$ and $k \geq 2$, let $n$ be the largest number of vertices in a $k$-coloring of a complete graph in which there is
\bi
\item no rainbow triangle,
\item no monochromatic $B_{m}$,
\item a G-partition with all parts having order at most $m - 1$, and
\item only one $m$-admissible color.
\ei
Then
$$
n \leq \begin{cases}
3m - 1 & \text{ if $k = 2$,}\\
5m - 5 & \text{ otherwise.}
\end{cases}
$$
\end{lemma}

\begin{proof}
First suppose $k = 2$ and let $G$ be a $2$-coloring of $K_{3m}$ with no monochromatic $B_{m}$, supposing that only one of the two colors is $m$-admissible. Note that $G$ trivially contains no rainbow triangle since it uses only $2$ colors. Say red is the $m$-admissible color and blue is the other ($m$-inadmissible) color. Let $u$ and $v$ be any two vertices of $G$ that are joined by a red edge. Then there are $3m - 2$ other vertices in $G$ but at most $m - 1$ of them can have blue edges to $u$ and at most $m - 1$ of them can have blue edges to $v$. This leaves at least 
$$
(3m - 2) - 2(m - 1) \geq m
$$
vertices with red edges to both $u$ and $v$, creating a red copy of $B_{m}$, for a contradiction.

Now suppose $k \geq 3$ and let $G$ be a $k$-coloring of $K_{5m - 4}$ with no monochromatic $B_{m}$ and suppose only one of the colors is $m$-admissible. Consider a G-partition with all parts having order at most $m - 1$ as provided by the hypothesis. Say red is the $m$-admissible color, so red must clearly appear in the G-partition, and let blue be the other color appearing in the G-partition. Choose two parts of this partition with red edges between them, say $H_{1}$ and $H_{2}$. Then since $|G| = 5m - 4$, there are at least $3m - 2$ remaining vertices in $G \setminus (H_{1} \cup H_{2})$. At most $m - 1$ of these remaining vertices can have blue edges to $H_{1}$ and at most $m - 1$ of these vertices can have blue edges to $H_{2}$. This means that at least $m$ vertices must have red edges to both $H_{1}$ and $H_{2}$, producing a red copy of $B_{m}$, for a contradiction.
\end{proof}



Let $\ell = \ell(m)$ be the number of colors that are $m$-inadmissible and define the quantity $gr_{k, \ell}(K_{3} : H)$ to be the minimum integer $n$ such that every $k$ coloring of $K_{n}$, with at least $\ell$ different $m$-inadmissible colors, contains either a rainbow triangle or a monochromatic copy of $H$. We may now state our main result, which provides a general upper bound on the Gallai-Ramsey numbers for any book with any number of colors.

\begin{theorem}\label{Thm:MainBm}
Given positive integers $k \geq 1$, $m \geq 3$, and $0 \leq \ell \leq k$, let
$$
gr_{k, \ell, m} = \begin{cases}
m + 2 - \ell & \text{ if $k = 1$,}\\
R_{m}' \cdot 5^{\frac{k-2}{2}}+1 - (m - 1)\ell & \text{ if $k$ is even,}\\
2 \cdot R_{m}' \cdot 5^{\frac{k-3}{2}}+1 - (m - 1)\ell & \text{ if $k \geq 3$ is odd.}
\end{cases}
$$
Then
$$
gr_{k, \ell}(K_{3} : B_{m}) \leq gr_{k, \ell, m}.
$$
\end{theorem}

\begin{proof}
For a contradiction, suppose $G$ is a $k$-coloring of $K_{n}$ with
$$
n = gr_{k, \ell, m}
$$
containing no rainbow triangle and no monochromatic copy of $B_{m}$ with the property that there are at least $\ell$ different $m$-inadmissable colors in $G$. We prove the result by induction on $k - \ell$. If $k = 1$, then Theorem~\ref{Thm:MainBm} is immediate so next suppose $k = 2$. Then for any $m \geq 3$, it is easy to verify that
$$
n = R_{m}' + 1 - \ell(m - 1) \geq R_{m}' + 1 - 2(m - 1) \geq R_{m},
$$
meaning that $G$ contains a monochromatic $B_{m}$, for a contradiction. In particular, if $m = 3$, we have $n = 19 > 14 = R_{3}$ and if $m = 4$, we have $n = 29 > 18 = R_{4}$.

For the remainder of the proof, we may assume $k \geq 3$, and so $n \geq 2 R_{m}' - 3(m - 1)$. If $k - \ell = 0$, then there are no $m$-admissible colors but with
$$
n \geq 2 R_{m}' - 3(m - 1) \geq 5m/2,
$$
by Theorem~\ref{Thm:Stars}, there is a vertex with at least $\frac{2n}{5} > m$ edges in a single color, so this color would be $m$-admissible, a contradiction.

Next suppose $k - \ell = 1$, and consider a G-partition. If there are only two parts, then either both parts have at least $m$ vertices or one part has fewer than $m$ vertices. If both parts have at least $m$ vertices, then neither part contains any edges in this color, reducing $k$ by $1$ and reducing the problem to the case $k - \ell = 0$. Applying induction on $k - \ell$ within a largest part yields the desired result. On the other hand, if one part has fewer than $m$ vertices, the vertices in this smaller part may be removed, yielding a graph with one more $m$-inadmissible color (within the larger part), again reducing the problem to the case $k - \ell = 0$. More generally, if there are more than two parts, none can have order at least $m$ since the only color with edges to a part of order at least $m$ must be $m$-admissible, yielding a bipartition as above. With all parts having order at most $m - 1$, but only one $m$-admissible color, Lemma~\ref{Lemma:1admissible} implies the existence of a monochromatic $B_{m}$ in the $m$-admissible color, to complete the case $k - \ell = 1$.

This means we may assume $k - \ell \geq 2$. 
There is a G-partition of $G$, say using colors red and blue. Consider such a G-partition with the smallest number of parts, say $t$. If $t \geq R_{m}$, then by Theorem~\ref{Thm:RamseyBm}, the reduced graph of this partition contains a monochromatic copy of $B_{m}$, a contradiction. We may therefore assume $t \leq R_{m} - 1$. 
Let $H_{1},H_{2},\dots,H_{t}$ be parts of such a G-partition.

First suppose $t \leq 3$. If $t = 3$, then the reduced graph is a triangle but this contains a bipartition with only one color on the edges between the parts so we may assume $t=2$. Let red be the color of the edges between the two parts. If $|H_{1}| \leq m - 1$, then red is $m$-inadmissible within $H_{2}$ so we remove $H_{1}$ from the graph and apply induction on $k - \ell$ within $H_{2}$. This means that $|H_{1}|, |H_{2}| \geq m$, but then there can be no red edges within either $H_{1}$ or $H_{2}$. By induction on $k - \ell$, this gives
\beqs
n & = & |H_{1}| + |H_{2}|\\
~ & \leq & 2 \left[ gr_{k - 1, \ell, m} - 1\right]\\
~ & < & gr_{k, \ell, m},
\eeqs
a contradiction. We may therefore assume $t \geq 4$ and, by the minimality of $t$, each part of the G-partition has incident edges in both colors that appear in the G-partition. Suppose red and blue are the two colors appearing in the G-partition.

If one of red or blue was $m$-inadmissible in $G$, then since both colors must induce a connected subgraph of the reduced graph, all parts must have order less than $m$, so by Lemma~\ref{Lemma:r=0}, we have $|G| \leq R_{m}'$.
This is a contradiction since $k \geq 3$, so red and blue must both be $m$-admissible within $G$.

Let $r$ be the number of parts of the G-partition with order at least $m$, say with
$$
|H_{1}|, |H_{2}|, \dots, |H_{r}| \geq m ~ ~ ~ ~ \text{and} ~ ~ ~ ~ |H_{r + 1}|, |H_{r + 2}|, \dots, |H_{t}| \leq m - 1.
$$

If $r \geq 4$ and $t \geq 6$, then any choice of $6$ parts containing the $4$ parts $\mathscr{H} = \{H_{1}, H_{2}, H_{3}, H_{4}\}$ will contain a monochromatic triangle in the reduced graph. Such a triangle must contain at least one part from $\mathscr{H}$, meaning that the corresponding subgraph of $G$ must contain a monochromatic copy of $B_{m}$, for a contradiction. Thus, we may assume either $4 \leq t \leq 5$ or $r \leq 3$. We consider cases based on the value of $r$. Note that the case $r = 0$ follows from Lemma~\ref{Lemma:r=0}.

The following cases are actually sharp in the sense that the remainder of this proof will be used (as is) to complete the proof of Theorem~\ref{Thm:GR-Bm-Small} in Section~\ref{Sec:Small}.

\setcounter{case}{0}
\begin{case}
$r = 1$.
\end{case}

Since $t \geq 4$, both red and blue are $m$-inadmissible within $H_{1}$ so if $|H_{1}| \geq n - 2(m - 1)$, we may simply apply induction on $k - \ell$, so suppose $|H_{1}| \leq n - 2(m - 1) - 1$. This means that there are at least $m$ vertices in $G \setminus H_{1}$ with the same color on all edges to $H_{1}$, say red. This implies that there are no red edges within $H_{1}$ so $|H_{1}| \leq gr_{k - 1, \ell + 1, m} - 1 < \frac{n}{2}$.

Let $H_{r}$ (and $H_{b}$) be the subsets of vertices of $G \setminus H_{1}$ that have all red (or blue respectively) edges to $H_{1}$. To avoid a red $B_{m}$, all edges between the parts contained in $H_{r}$ must be blue and similarly all edges between the parts contained in $H_{b}$ must be red. At this point, we observe an easy fact.

\begin{fact}\label{Fact:4mMonochi}
For any $m \geq 3$ and $n' \geq 3m - 2$, in any G-coloring of a $K_{n'}$ in which only one color appears between the parts and all parts have order at most $m - 1$, there is a monochromatic copy of $B_{m}$.
\end{fact}

Therefore, by Fact~\ref{Fact:4mMonochi}, we have $|H_{r}|, |H_{b}| \leq 3m - 3$, so $|G \setminus H_{1}| \leq 6m - 6$ and
$$
gr_{k, \ell, m} - (gr_{k - 1, \ell + 1, m} - 1) \leq 6m - 6.
$$
More specifically, if $|H_{b}| \geq m$, then $|H_{1}| < gr_{k - 2, \ell, m}$ so
$$
gr_{k, \ell, m} - (gr_{k - 2, \ell, m} - 1) \leq 6m - 6,
$$
which means that $k \leq 2$, a case that has already been considered.
On the other hand, if $|H_{b}| \leq m - 1$, then
$$
gr_{k, \ell, m} - (gr_{k - 1, \ell + 1, m} - 1) \leq 4m - 4.
$$
This again implies that $k \leq 2$, a case that has already been considered, completing the proof in this case.

\begin{case}
$r = 2$.
\end{case}

Suppose red is the color of the edges between $H_{1}$ and $H_{2}$ so neither $H_{1}$ nor $H_{2}$ can contain red edges and recall that blue must be $m$-inadmissible within $H_{1}$ and $H_{2}$. If there are at most $2(m - 1)$ vertices in $G \setminus (H_{1} \cup H_{2})$, then it is easy to show that $n < gr_{k, \ell}(K_{3} : B_{m})$ for a contradiction, so we may assume $|G \setminus (H_{1} \cup H_{2})| \geq 2(m - 1) + 1$.

For $i \in \{1, 2\}$, let $A_{i}$ be the set of vertices in $G \setminus (H_{1} \cup H_{2})$ with red edges to $H_{i}$. Note that $A_{1} \cap A_{2} = \emptyset$ since any intersection would create a red copy of $B_{m}$. Then for $i \in \{1, 2\}$, let $C_{i} = G \setminus (H_{1} \cup H_{2} \cup A_{3 - i})$, so $C_{i}$ has all edges to $H_{3 - i}$ in blue. Since $|G \setminus (H_{1} \cup H_{2})| \geq 2(m - 1) + 1$, at least one of $C_{1}$ or $C_{2}$ contains at least $m$ vertices, say $|C_{1}| \geq m$. Then $H_{2}$ contains no blue edges.

To avoid a blue copy of $B_{m}$, for each $i \in \{1, 2\}$, there can be no blue edges within $C_{i}$, meaning that all edges between the parts of the G-partition within $C_{i}$ are red, so by Fact~\ref{Fact:4mMonochi}, we know $|C_{i}| \leq 4m - 6$ for each $i \in \{1, 2\}$. On the other hand, if $|C_{2}| \geq m$, then $H_{1}$ also contains no blue edges and must therefore be smaller. This means that
\beqs
|G| & \leq & |H_{1}| + |H_{2}| + |C_{1} \cup C_{2}|\\
~ & \leq & [gr_{k - 1, \ell + 1, m} - 1] + [gr_{k - 2, \ell, m} - 1] + (4m - 6) + (m - 1)\\
~ & < & gr_{k, \ell, m},
\eeqs
a contradiction for $k \geq 3$, completing the proof in this case.

\begin{case}
$r = 3$.
\end{case}

To avoid a monochromatic $B_{m}$, the triangle in the reduced graph corresponding to the parts $\{H_{1}, H_{2}, H_{3}\}$ must not be monochromatic. Without loss of generality, suppose the edges from $H_{1}$ to $H_{2}$ are red and all other edges between these parts are blue. Then $H_{1}$ and $H_{2}$ contain no red or blue edges while $H_{3}$ contains no blue edges and red is $m$-inadmissible within $H_{3}$.

If there is a vertex $v \in G \setminus (H_{1} \cup H_{2} \cup H_{3})$ with blue edges to $H_{3}$, then to avoid creating a blue $B_{m}$, $v$ must have all red edges to both $H_{1}$ and $H_{2}$, producing a red $B_{m}$. This means that all edges between $H_{3}$ and $G \setminus (H_{1} \cup H_{2} \cup H_{3})$ must be red. If $|G \setminus (H_{1} \cup H_{2} \cup H_{3})| \leq m - 1$, then
\beqs
|G| & \leq & |H_{1}| + |H_{2}| + |H_{3}| + (m - 1)\\
~ & \leq & 2[gr_{k - 2, \ell, m} - 1] + [gr_{k - 1, \ell + 1, m} - 1] + (m - 1)\\
~ & < & gr_{k, \ell, m},
\eeqs
a contradiction, so suppose $|G \setminus (H_{1} \cup H_{2} \cup H_{3})| \geq m$. Then $H_{3}$ contains no red edges and all edges between parts within $G \setminus (H_{1} \cup H_{2} \cup H_{3})$ must be blue. By Fact~\ref{Fact:4mMonochi}, we get
\beqs
|G| & \leq & |H_{1}| + |H_{2}| + |H_{3}| + (4m - 6)\\
~ & \leq & 3[gr_{k - 2, \ell, m} - 1] + (4m - 6)\\
~ & < & gr_{k, \ell, m},
\eeqs
again a contradiction, completing the proof in this case.

\begin{case}
$r \geq 4$.
\end{case}

As observed previously, this implies that $4 \leq t \leq 5$. Looking only at the subgraph of the reduced graph induced on the $r$ large parts, there can be no monochromatic triangle. If $r = 5$, there is only one coloring of $K_{5}$ with no monochromatic triangle and if $r = 4$, there are two colorings of $K_{4}$ with no monochromatic triangle. In all of these colorings, every vertex has an incident edge in both colors, meaning that all of the $r$ corresponding parts of order at least $m$ must have no red or blue edges. Then
\beqs
|G| & \leq & \sum_{i = 1}^{t} |H_{i}|\\
~ & \leq & 5 [gr_{k - 2, \ell, m} - 1]\\
~ & < & gr_{k, \ell, m},
\eeqs
a contradiction, completing the proof of this case, and therefore the proof of Theorem~\ref{Thm:MainBm}.
\end{proof}

\section{Some Small Cases}\label{Sec:Small}

In this section, we provide the sharp Gallai-Ramsey number for several small books. The proof of this result follows the same outline as the proof of Theorem~\ref{Thm:MainBm} except each step is improved in order to produce the sharp result.

In place of Lemma~\ref{Lemma:r=0}, we provide the following lemma. For given values of $k, \ell, m$, define the quantity
$$
gr_{k, \ell, m} = \begin{cases}
m + 2 - \ell & \text{if $k = 1$},\\
(R_{m} - 1) \cdot 5^{(k - 2)/2} + 1 - \ell (m - 1) & \text{if $k$ is even,}\\
2 \cdot (R_{m} - 1) \cdot 5^{(k - 3)/2} + 1 - \ell (m - 1) & \text{if $k \geq 3$ is odd.}
\end{cases}
$$

\begin{lemma}\label{Lemma:Smallr=0}
Let $k, \ell, m$ be integers with $k \geq 3$, $0 \leq \ell \leq k - 2$, and $2 \leq m \leq 5$. If $G$ is a G-coloring of $K_{p}$ with
$$
p \geq gr_{k, \ell, m}
$$
using $k$ colors in which all parts of a G-partition have order at most $m - 1$ and $\ell$ colors are $m$-inadmissible, then $G$ contains a monochromatic copy of $B_{m}$.
\end{lemma}

\begin{proof}

For the values of $m$ in question, we have
$$
p \geq \begin{cases}
18 & \text{ if $m = 2$,}\\
25 & \text{ if $m = 2$,}\\
32 & \text{ if $m = 4$, and}\\
37 & \text{ if $m = 5$.}\\
\end{cases}
$$
Let $t$ be the number of parts in the partition.

When $m = 2$, all parts of the assumed G-partition have order $1$, meaning that $G$ is simply a $2$-coloring. Since $|G| = p > 10 = R_{2}$, the claim is immediate. We consider cases for the remaining values of $m$.

\setcounter{case}{0}
\begin{case}
$m = 3$.
\end{case}

With $p \geq 25$ vertices and all parts of the partition having order at most $2$, there must be at least $13$ parts in this partition. If there were at least $R_{3} = 14$ parts in the partition, then there would be a monochromatic $B_{3}$ in the reduced graph, which means we may assume there are exactly $13$ parts in the G-partition. In particular, this means that $12 > R_{2}$ parts have order $2$, so the reduced graph on this set of parts of order $2$ contains a monochromatic copy of $B_{2}$. The parts corresponding to this monochromatic copy of $B_{2}$ induce a monochromatic $B_{3}$ (in fact, $B_{4}$) in $G$, as desired.

\begin{case}
$m = 4$.
\end{case}

With $p \geq 32$ vertices and all parts of the partition having order at most $3$, there must be at least $11$ parts in this partition. On the other hand, since the reduced graph of this partition is a $2$-colored copy of $K_{t}$ and $R_{4} = 18$, we see that $t \leq 17$. If there is a copy of $B_{2}$ in the reduced graph of the parts of order at least $2$, this would produce a monochromatic copy of $B_{4}$ in $G$. This means that at most $R_{2} - 1 = 9$ parts can have have order at least $2$. Since at most $3\cdot 9 = 27$ vertices can be in parts of order at least $2$, there must be at least $5$ parts of order $1$, making a total of at least $14$ parts, so $14 \leq t \leq 17$.

If there were at most $5$ parts of order $3$, then since there can be at most $9$ parts of order at least $2$, we see that $p \leq 5\cdot 3 + 4\cdot 2 + 8 = 31$, a contradiction. Thus, there are at least $6$ parts of order $3$, meaning that there is a monochromatic triangle within the reduced graph restricted to the vertices corresponding to parts of order $3$, say corresponding to parts $H_{1}, H_{2}, H_{3}$ each of order $3$ with blue edges.

If any vertex in $G' = G \setminus (H_{1} \cup H_{2} \cup H_{3})$ has blue edges to two sets $H_{i}$ and $H_{j}$, then there is a blue $B_{4}$, meaning that every vertex in $G'$ has blue edges to at most one of the sets $H_{i}$. There is therefore a set of at least
$$
\ceil{ \frac{ |G'| }{3} } = \ceil{ \frac{32 - 9}{3} } = 8
$$
vertices $G^{\pp} \subseteq G'$ with all red edges to two parts, say $H_{1}$ and $H_{2}$. See Figure~\ref{Fig:Gpp} where the darker edges are blue and the lighter edges are red. Within $G^{\pp}$, any red edge would produce a red $B_{4}$ (actually $B_{6}$), meaning that all edges between the parts of the G-partition in these vertices must be blue. By Lemma~\ref{Lemma:1admissible} there are at most $10$ such vertices so there exist two vertices $x, y \in G \setminus (H_{1} \cup H_{2} \cup H_{3} \cup G^{\pp})$. In fact, it is easy to see that there must be exactly $3$ parts of the G-partition within $G^{\pp}$, meaning that $|G^{\pp}| \leq 9$ and at least two parts have order $3$ with one having order either $2$ or $3$.

\begin{figure}[H]
\begin{center}
\includegraphics{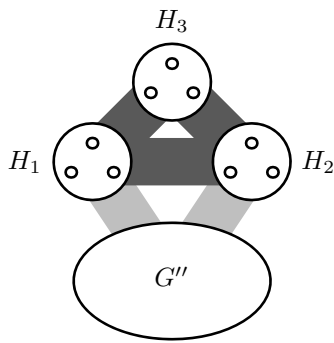} 
\end{center}
\caption{The structure of $G^{\pp}$ \label{Fig:Gpp}}
\end{figure}

As observed previously, $x$ and $y$ must each have red edges to at least one of $H_{1}$ or $H_{2}$. If either $x$ or $y$ has at least $4$ red edges to $G^{\pp}$, then there is a red $B_{4}$ using, for example, $x$, a vertex of $H_{1}$, and $4$ vertices from $G^{\pp}$, so $x$ and $y$ must each have at most $3$ red edges to $G^{\pp}$. On the other hand, if $x$ or $y$ has all blue edges to $G^{\pp}$, it is easy to see there is a blue $B_{4}$ within say $G^{\pp} \cup \{x\}$, so $x$ and $y$ must each have red edges to at least one vertex of $G^{\pp}$. Since $G^{\pp}$ consists of only $3$ parts of the G-partition and each has at least $2$ vertices, $x$ and $y$ must each have at least $2$ red edges to $G^{\pp}$.

If $x$ has red edges to a part of order $3$ in $G^{\pp}$, say $H_{4}$, then one vertex from each of the other two parts in $G^{\pp}$ form the spine of a $B_{4}$ with $x$ and $H_{4}$ as the pages. Thus, $x$ and $y$ must both have red edges to a part of order $2$ in $G^{\pp}$, say $H_{5}$. Then one vertex from each of the other two parts of $G^{\pp}$ form the spine of a blue copy of $B_{4}$ with $x$, $y$, and $H_{5}$ as the pages, completing the proof of the case $m = 4$.

\begin{case}
$m = 5$.
\end{case}
Note that $t\leq 20$ since $R(B_{5},B_{5})=21$. Let $a,b,c$ and $d$ be the number of parts of the G-partition of order $4,3,2$ and $1$ respectively. Let $H_{1},H_{2},\dots,H_{t}$ be the parts of this G-partition. We first present several structural claims before breaking the remainder of the proof into subcases based on the value of $a$.

\begin{claim}\label{Clm:5.1}
If $|H_{i}|\geq 3$, say for $i=1,2$, with all blue edges between $H_{1}$ and $H_{2}$, then there are at most $4$ other parts with all blue edges to $H_{1}\cup H_{2}$ and the total vertices in these parts is most $4$. If there are at least $2$ other parts with all red edges to $H_{1}\cup H_{2}$, then all edges between any pair of these parts must be blue.
\end{claim}

\begin{proof}
Since the edges between $H_{1}$ and $H_{2}$ are blue and each part has at least one vertex, to avoid creating a blue copy of $B_{5}$, there are most $4$ vertices with all blue edges to $H_{1}\cup H_{2}$ and so at most $4$ parts of the partition with blue edges to $H_{1} \cup H_{2}$. If there are at least $2$ parts, each with all red edges to $H_{1}\cup H_{2}$, to avoid creating a red copy of $B_{5}$, the color between any two such parts must be blue.
\end{proof}

\begin{claim}\label{Clm:5.2}
If $|H_{i}| = 4$ for $1 \leq i \leq 3$ and the corresponding reduced graph of these parts is a monochromatic triangle, say in blue, then there are at most two other parts with blue edges to $H_{i}$ for each $i$.
\end{claim}

\begin{proof}
Suppose there are three parts with blue edges to some part $H_{i}$ with $i \leq 3$, say $i = 1$ and let $H_{4}, H_{5}$, and $H_{6}$ be these parts. To avoid creating a blue copy of $B_{5}$, all edges from $H_{4} \cup H_{5} \cup H_{6}$ to $H_{2} \cup H_{3}$ must be red. Any red edge between the parts in $\{ H_{4}, H_{5}, H_{6}\}$ would create a red copy of $B_{5}$ so all edges between these parts must be blue. Then a blue spine edge $uv$ between $H_{4}$ and $H_{5}$ along with all blue edges from $\{ u, v\}$ to $H_{1} \cup H_{6}$ provides a monochromatic (blue) copy of $B_{5}$, for a contradiction.
\end{proof}

\begin{claim}\label{Clm:5.3}
If $|H_{1}|=|H_{2}|=4$, say with all blue edges between $H_{1}$ and $H_{2}$, then there are at most $5$ other parts with different colors to $H_{1}\cup H_{2}$.
\end{claim}

\begin{proof}
Suppose there are at least $6$ parts with different colors to $H_{1} \cup H_{2}$, and let $H_{i}$ for $3 \leq i \leq 8$ be these parts. Since $R(K_{3},K_{3})=6$, the reduced graph of these six parts contains a monochromatic triangle $T$ corresponding to three parts, say $H_{3}, H_{4}$ and $H_{5}$. At least two of the parts in $\{ H_{3}, H_{4}, H_{5}\}$ must have red edges to one of $H_{1}$ or $H_{2}$, and at least two of the parts in $\{ H_{3}, H_{4}, H_{5}\}$ must have blue edges to $H_{1}$ or $H_{2}$. Regardless whether $T$ is red or blue, there will be a monochromatic copy of $B_{5}$, completing the proof of Claim~\ref{Clm:5.3}. For example, if the triangle is red, suppose $H_{3}$ and $H_{4}$ have red edges to $H_{1}$. Then a red spine edge $uv$ from $H_{3}$ to $H_{4}$ along with all red edges to $H_{5} \cup H_{1}$ produces the desired red copy of $B_{5}$.
\end{proof}

\begin{claim}\label{Clm:5.4}
If $|H_{1}|=4$ and $|H_{2}|=3$ for two parts of the partition, then regardless of the color on edges between $H_{1}$ and $H_{2}$, there are at most $6$ parts with all blue (or red) edges to $H_{1}$ and all red (respectively blue) edges to $H_{2}$, and there are a total of at most $9$ parts with different colors to $H_{1}\cup H_{2}$.
\end{claim}

\begin{proof}
First suppose there are at least $7$ parts with all blue edges to $H_{1}$ and all red edges to $H_{2}$. By Fact~\ref{Fact:B1B2}, there is either a blue triangle (a copy of $B_{1}$) or a red copy of $B_{2}$ in the reduced graph restricted to these $7$ parts. Either of these, along with all edges to $H_{1}$ or $H_{2}$ (respectively) would produce a monochromatic copy of $B_{5}$, for a contradiction. There can therefore be at most $6$ parts with all blue (or red) edges to $H_{1}$ and all red (respectively blue) edges to $H_{2}$.

Now suppose there are at least $10$ parts, each with different colors on edges to $H_{1}$ and $H_{2}$. By Theorem~\ref{Thm:RamseyBm}, there is a monochromatic copy of $B_{2}$ within the reduced graph induced on these $10$ parts. Without loss of generality, suppose this copy of $B_{2}$ is blue, let the spine be the edges between $H_{3}$ and $H_{4}$ and let the remaining vertices of the copy of $B_{2}$ represent parts $H_{5}$ and $H_{6}$. If both $H_{3}$ and $H_{4}$ have blue edges to either $H_{1}$ or $H_{2}$, then there would be a blue copy of $B_{5}$ for a contradiction, so this cannot be the case. Without loss of generality, suppose $H_{3}$ (respectively $H_{4}$) has blue (respectively red) edges to $H_{1}$ and red (respectively blue) edges to $H_{2}$. If $H_{5}$ or $H_{6}$ has blue edges to $H_{1}$, then the blue edges between $H_{3}$ and $H_{5}$ or $H_{6}$ would serve as the spine of a blue copy of $B_{5}$, for a contradiction. This means that all edges between $H_{5} \cup H_{6}$ and $H_{1}$ are red so all edges between $H_{5} \cup H_{6}$ and $H_{2}$ are blue. If the edges between $H_{5}$ and $H_{6}$ were blue, these would serve as a spine for a blue copy of $B_{5}$, meaning that the edges between $H_{5}$ and $H_{6}$ must be red.

We now turn our attention to the remaining $10 - 4 = 6$ parts, say these are $\{H_{7}, H_{8}, \dots, H_{12}\}$. If any of these parts, say $H_{7}$, has blue edges to $H_{4}$, then $H_{7}$ must have red edges to both $H_{5}$ and $H_{6}$ to avoid making a blue copy of $B_{5}$. But this makes a red copy of $B_{5}$ with the spine being an edge between $H_{5}$ and $H_{6}$, meaning that all edges from $H_{4}$ to $H_{7} \cup \dots \cup H_{12}$ must be red. By the first part of this proof, there are at most $6$ parts with blue edges to $H_{1}$ so one of $\{H_{7}, H_{8}, \dots, H_{12}\}$ must have red edges to $H_{1}$, suppose $H_{7}$.

If $H_{7}$ has a red edge to any part in $\{H_{8}, \dots, H_{12}\}$, say $H_{8}$, then a red edge from $H_{7}$ to $H_{4}$ along with all red edges to $H_{1} \cup H_{8}$ produces a red copy of $B_{5}$, meaning that all edges from $H_{7}$ to $H_{8} \cup \dots \cup H_{12}$ must be blue. If there was another set, say $H_{8}$, with red edges to $H_{1}$, then the same argument shows that all edges from $H_{8}$ to $H_{9} \cup \dots \cup H_{12}$ must be blue. Then a blue edge from $H_{7}$ to $H_{8}$ would form the spine of a copy of $B_{5}$ using all blue edges to $H_{9} \cup H_{10} \cup H_{2}$, meaning that all edges from $H_{1}$ to $H_{8} \cup \dots \cup H_{12}$ must be blue so all edges from $H_{2}$ to these parts must be red. Any blue edge between parts in $\{H_{8}, \dots, H_{12}\}$ would form the spine of a $B_{5}$ with blue edges to $H_{1} \cup H_{7}$, so all edges between these parts must be red. Then these edges along with all red edges to $H_{2} \cup H_{4}$ easily produce a red copy of $B_{5}$, completing the proof of Claim~\ref{Clm:5.4}.
\end{proof}

Recall the definition of $a, b, c$, and $d$ to be the number of parts of the G-partition of order $4, 3, 2$, and $1$ respectively. 

\begin{claim}\label{Clm:5.7}
$a+b\leq 9$ and $a+b+c\leq 12$.
\end{claim}

\begin{proof}
Since $R(B_{2},B_{2})=10$, the first bound is immediate. Since $R(B_{3},B_{3})=14$, we immediately see that $a + b + c \leq 13$ but if $a + b + c = 13$, we may consider the reduced graph on the corresponding parts along with one part of order $1$. This reduced graph contains a monochromatic copy of $B_{3}$. Regardless of if and where the part of order $1$ appears as a vertex in this copy of $B_{3}$, since all other parts have order at least $2$, this produces a monochromatic copy of $B_{5}$ in $G$.
\end{proof}

Since $|G| \geq 37$, it must be the case that $4a + 3b + 2c + d \geq 37$. Recall that $a + b + c + d = t \leq 20$. First an easy fact that will help to simplify the cases.

\begin{fact}\label{Fact:MonoChi}
If all edges between parts of a G-partitioned complete graph $K_{n}$ have one color, there are at least $4$ parts, and $n \geq 9$, then this contains a monochromatic copy of $B_{5}$. Note that $8$ vertices does not suffice for this result since $4$ parts, each of order $2$, does not force a monochromatic copy of $B_{5}$.
\end{fact}

The remainder of the proof is broken into cases based on the value of $a$.

\begin{subcase}
$a=9$.
\end{subcase}
If $a=9$, then since $n = 37$, we have $b=c=0$ and $d=1$. Since $R(B_{2},B_{2})=10$, there must be a monochromatic copy of $B_{2}$ in the reduced graph. At least three of the parts used in this copy of $B_{2}$ have order $4$, so this forms a monochromatic copy of $B_{5}$.

\begin{subcase}
$6\leq a\leq 8$.
\end{subcase}
As $R(K_{3},K_{3})=6$, there must be a monochromatic triangle in the reduced graph induced on the parts of order $4$, say in blue. Call the corresponding parts $H_{1}$, $H_{2}$, and $H_{3}$. By Claim~\ref{Clm:5.2}, there are at most two other parts with blue edges to $H_{i}$ for each $i$. If these parts exist, name parts $H_{4}$ and $H_{5}$ with blue edges to $H_{1}$ and red edges to $H_{2} \cup H_{3}$, and name parts $H_{6}$ and $H_{7}$ with blue edges to $H_{2}$ and with red edges to $H_{1} \cup H_{3}$. By Claim~\ref{Clm:5.1}, there is no part with blue edges to all of $H_{1} \cup H_{2}$ except $H_{3}$. 

If $a=8$, then since $a+b\leq 9$, we have $b\leq 1$. If $b=0$, the total number of vertices in $\{H_{1},H_{2},H_{3},H_{4},H_{5},H_{6},H_{7}\}$ is at most $28$, there must be at least $9$ vertices with red edges to $H_{1}\cup H_{2}$ and the corresponding number of parts must be at least $5$. By Claim~\ref{Clm:5.1}, all edges between these parts must be blue. By Fact~\ref{Fact:MonoChi}, this set contains a blue copy of $B_{5}$, a contradiction. If $b=1$, then since there are only two vertices remaining in $G$, we have either $c=0$ and $d=2$ or $c=1$ and $d=0$. For the first situation, there are at least $9$ vertices with red edges to $H_{1}\cup H_{2}$ and the corresponding number of parts must be at least $4$. By Claim~\ref{Clm:5.1}, all edges between these parts are blue so by Fact~\ref{Fact:MonoChi}, this set contains a blue $B_{5}$, a contradiction. For the second situation, since $a=8$, $b=1$, $c=1$ and $R(B_{2},B_{2})=10$, there must be a monochromatic $B_{2}$ with at most one part of order $2$ and the other three parts having order at least $3$, yielding a monochromatic copy of $B_{5}$, a contradiction.

If $a=7$, then since $a+b\leq 9$, we have $b\leq 2$. The total number of vertices in $\{H_{1},H_{2},H_{3},H_{4},H_{5},H_{6},H_{7}\}$ is at most $28$, there are at least $9$ vertices with red edges to $H_{1}\cup H_{2}$ and the corresponding number of parts must be at least $4$. By Claim~\ref{Clm:5.1}, all edges between these parts must be blue. By Fact~\ref{Fact:MonoChi}, this set contains a blue copy of $B_{5}$, a contradiction. Similarly if $a=6$, then since $a+b\leq 9$, we have $b\leq 3$. The total number of vertices in $\{H_{1},H_{2},H_{3},H_{4},H_{5},H_{6},H_{7}\}$ is at most $27$, there are at least $10$ vertices with red edges to $H_{1}\cup H_{2}$ and the corresponding number of parts must be at least $4$. By Claim~\ref{Clm:5.1}, all edges between these parts must be blue so by Fact~\ref{Fact:MonoChi}, this set contains a blue copy of $B_{5}$, a contradiction.

\begin{subcase}
$2\leq a\leq 5$.
\end{subcase}
If $2\leq a\leq 5$, we choose two parts of order $4$, say $H_{1}$ and $H_{2}$ with blue edges between them. We consider cases based on the number of parts with blue edges to $H_{1} \cup H_{2}$.

First suppose there is no part with blue edges to $H_{1}\cup H_{2}$. By Claim~\ref{Clm:5.3}, there are at most $5$ parts with different colors to $H_{1}\cup H_{2}$, say $H_{3},H_{4},H_{5},H_{6},H_{7}$. The total number of vertices in $\{H_{1}, H_{2}, \dots, H_{7}\}$ is at most $26$, there are at least $11$ vertices with red edges to $H_{1}\cup H_{2}$. By Claim~\ref{Clm:5.1} and Fact~\ref{Fact:MonoChi}, there is a blue copy of $B_{5}$, for a contradiction.

Next suppose there is exactly one part with blue edges to $H_{1}\cup H_{2}$, say $H_{3}$. We first claim that there are at most $2$ parts with blue edges to $H_{1}$ and red edges to $H_{2}$, say $H_{4}$ and $H_{5}$. Indeed, if there was another such part, say $H'$, with blue edges to $H_{1}$ and red edges to $H_{2}$, then to avoid a blue copy of $B_{5}$, the edges from $H_{3}$ to $\{H_{4},H_{5},H'\}$ must all be red. To avoid a red copy of $B_{5}$, all edges between any two parts in $\{H_{4},H_{5},H'\}$ must be blue. Then a blue spine edge between $H_{4}$ and $H_{5}$ along with all blue edges to $H_{1} \cup H'$ creates a blue copy of $B_{5}$, for a contradiction meaning that $H'$ cannot exist. Similarly, there are most $2$ parts with red edges to $H_{1}$ and blue edges to $H_{2}$, say $H_{6}$ and $H_{7}$. The sum of vertices in $\{H_{1}, H_{2}, \dots, H_{7}\}$ is most $26$, so there are at least $11$ vertices with all red edges to $H_{1}\cup H_{2}$ and the corresponding number of parts must be at least $4$. By Claim~\ref{Clm:5.1} and Fact~\ref{Fact:MonoChi}, there is a blue copy of $B_{5}$, a contradiction.

Next suppose there are $t$ parts, where $2 \leq t \leq 4$, each with all blue edges to $H_{1}\cup H_{2}$, say $H_{3}$, $H_{4}$, $H_{5}$, and $H_{6}$. By Claim~\ref{Clm:5.1}, all edges between parts in $\{H_{3}, H_{4}, H_{5}, H_{6}\}$ are red and to avoid a blue copy of $B_{5}$, there are at most $4$ total vertices in $H_{3} \cup H_{4} \cup H_{5} \cup H_{6}$. We claim that there are at most $2$ parts with blue edges to $H_{1}$ and red edges to $H_{1}$, say $H_{7}$ and $H_{8}$, and the total number of vertices in $H_{7}\cup H_{8}$ is at most $4$. Supposing there is a third such part, $H'$, in order to avoid a blue copy of $B_{5}$, all edges between $H_{7} \cup H_{8} \cup H'$ and $H_{2} \cup H_{3} \cup \dots \cup H_{6}$ must be red. To avoid a red copy of $B_{5}$, all edges between the parts in $\{H_{5}, H_{6}, H'\}$ must be blue, yielding a blue copy of $B_{5}$ with $H_{1}$, a contradiction meaning that $H'$ cannot exist. Since all edges between $H_{3}$ and $H_{4}$ are red, there can be at most $4$ vertices in $H_{7} \cup H_{8}$ to avoid creating a red copy of $B_{5}$. Similarly, there can be at most $2$ parts with red edges to $H_{1}$ and blue edges to $H_{1}$, say $H_{9}$ and $H_{10}$, and the total number of vertices in $H_{7}\cup H_{8}$ is at most $4$. The sum of vertices in $\{H_{1}, H_{2}, \dots, H_{10}\}$ is at most $20$ so there are at least $17$ vertices with all red edges to $H_{1}\cup H_{2}$ and the corresponding number of parts must be at least $4$. By Claim~\ref{Clm:5.1} and Fact~\ref{Fact:MonoChi}, there is a blue copy of $B_{5}$, a contradiction.

\begin{subcase}
$a=1$.
\end{subcase}

With $a = 1$, since $n = 37$, $t \leq 20$, and $a + b + c \leq 12$, we must have $b \geq 3$. Choose two parts, say $H_{1}$ and $H_{2}$, with order of $4$ and $3$ respectively, and suppose all edges between them are blue.

First suppose there is no part with blue edges to $H_{1}\cup H_{2}$. Let $C$ denote the set of parts with blue edges to $H_{1}$ and red edges to $H_{2}$, and let $D$ denote the set of parts with red edges to $H_{1}$ and blue edges to $H_{2}$. By Claim~\ref{Clm:5.4}, we see that $|C|\leq 6$, $|D|\leq 6$, and $|C \cup D|\leq 9$. Let $C \cup D = \{H_{3}, \dots, H_{11}\}$ be this set of parts and note that some of these parts may not exist. We first show that there are at least $4$ parts of order $1$ within $C \cup D$. Since $R(K_{3},K_{3})=6$, for any six parts in $C \cup D$, say $\{H_{3}, \dots, H_{8}\}$, there is a monochromatic triangle, say with corresponding parts $H_{3}$, $H_{4}$ and $H_{5}$. By the pigeonhole principle, there are at least two parts with the same color to $H_{1}\cup H_{2}$, say $H_{4}$ and $H_{5}$ with blue edges to $H_{1}$ and red edges to $H_{2}$, and $H_{3}$ with red edges to $H_{1}$ and blue edges to $H_{2}$. To avoid a blue copy of $B_{5}$, the triangle must be red, and if $H_{3}$ has order at least $2$, the parts $\{H_{3},H_{4},H_{5}\}$ along with $H_{2}$, yields a red $B_{5}$. Thus, $H_{3}$ must have order $1$. Then repeat this argument to find a monochromatic triangle from the reduced graph induced on $(C \cup D) \setminus \{H_{3}\}$, and so on. There are therefore at least $4$ parts of order $1$ within $C \cup D$. Since $a+b\leq 9$ and $a=1$, $b\leq 8$, the total number of vertices of $\{H_{1}, \dots, H_{11}\}$ is at most $26$, so there are at least $11$ vertices remaining with red edges to $H_{1}\cup H_{2}$ and the corresponding number of parts is at least $4$. By Claim~\ref{Clm:5.1} and Fact~\ref{Fact:MonoChi}, there is a blue copy of $B_{5}$, a contradiction.

Next suppose there is only one part with all blue edges to $H_{1} \cup H_{2}$, say $H_{3}$. First an easy observation.
\begin{fact}\label{Fact:3.4.2}
There is no part (other than $H_{1}$) with blue edges to both $H_{2}$ and $H_{3}$. There is at most one vertex (other than $H_{2}$) with blue edges to both $H_{1}$ and $H_{3}$.
\end{fact}
We first show that there are at most three parts, say $\mathcal{A}$, with red edges to $H_{1}$ and blue edges to $H_{2}$, and additionally that there are at most $6$ vertices in these parts. For a contradiction, suppose $\mathcal{A} = \{H_{4}, H_{5}, H_{6}, H_{7}\}$. By Fact~\ref{Fact:3.4.2}, all edges from $H_{3}$ to $\mathcal{A}$ must be red. This means there can be no red edges between parts in $\mathcal{A}$ so the reduced graph induced on these parts is a blue $K_{4}$. Since all edges from $H_{2}$ to this set are blue, this yields a blue copy of $B_{5}$, meaning that $H_{7}$ cannot exist. Furthermore, if any of the parts in $\{H_{4}, H_{5}, H_{6}\}$ has order at least $2$, there would again be a blue copy of $B_{5}$, meaning that there can be at most $6$ vertices with red edges to $H_{1}$ and blue edges to $H_{2}$. Note that if there are $6$ such vertices, they must occur in only two parts. We next show that there are at most $5$ parts, say $\mathcal{B}$, with blue edges to $H_{1}$ and red edges to $H_{2}$, for a total of at most $15$ vertices in these parts. Indeed, if there were $6$ such parts, then the reduced graph induced on these parts would contain a monochromatic triangle which, along with appropriately colored edges to $H_{1}, H_{2}$, and / or $H_{3}$, produces a monochromatic copy of $B_{5}$.
If there were at least $9$ vertices remaining outside $H_{1} \cup H_{2} \cup H_{3} \cup V(\mathcal{A}) \cup V(\mathcal{B})$, these vertices must have all red edges to $H_{1} \cup H_{2}$ and so, by Fact~\ref{Fact:MonoChi}, the proof would be complete. Note that $H_{3}$ can have blue edges to at most one vertex of $\mathcal{B}$, say $v$, to avoid creating a blue copy of $B_{5}$. Thus, if $|H_{3}| \geq 2$, then there can be no red edge within $\mathcal{B} \setminus v$, meaning that there can be at most $6$ vertices in $\mathcal{B} \setminus v$ (in two parts), so we get
$$
|H_{1} \cup H_{2} \cup H_{3} \cup V(\mathcal{A}) \cup V(\mathcal{B})| \leq 4 + 3 + 3 + 6 + 7 = 22
$$
so there are at least $15$ vertices remaining outside, completing the proof. This means that $|H_{3}| = 1$ so
$$
|H_{1} \cup H_{2} \cup H_{3} \cup V(\mathcal{A}) \cup V(\mathcal{B})| \leq 4 + 3 + 1 + 6 + 15 = 29
$$
so there are at least $8$ vertices remaining outside. Since $9$ such vertices would complete the proof, we may assume equality holds so $\mathcal{A}$ includes $2$ parts of order $3$ each and $\mathcal{B}$ includes $5$ parts of order $3$ each for which the reduced graph is the unique $2$-coloring of $K_{5}$ with no monochromatic triangle. See Figure~\ref{Fig:Sketch} where the dark edges represent all blue edges between the parts and the light edges represent all red edges between the parts. By Fact~\ref{Fact:3.4.2}, all edges from $H_{3}$ to $\mathcal{B}$ must be red. To avoid a red copy of $B_{5}$, a part $H_{4} \in \mathcal{A}$ can have red edges to at most two parts in $\mathcal{B}$, meaning that $H_{4}$ must have blue edges to at least $3$ parts in $\mathcal{B}$. These edges must therefore form a blue triangle. This triangle along with blue edges from $\mathcal{B}$ to $H_{1}$, produces a blue copy of $B_{5}$.

\begin{figure}
\begin{center}
\includegraphics[width=2in]{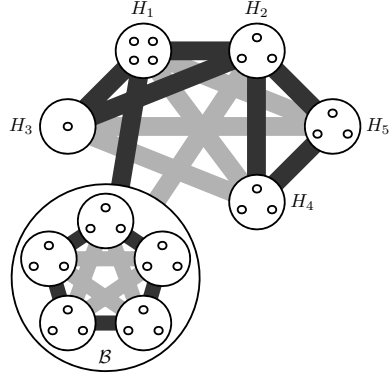}
\caption{The structure of the coloring \label{Fig:Sketch}}
\end{center}
\end{figure}

Next suppose there are at least two parts with blue edges to $H_{1} \cup H_{2}$, say $\mathcal{C}$. By Claim~\ref{Clm:5.1}, there are at most $4$ vertices in $\mathcal{C}$ and all edges between parts in $\mathcal{C}$ must be red to avoid a blue copy of $B_{5}$. As in the previous argument, there can be at most $3$ parts, say $\mathcal{A}$, with red edges to $H_{1}$ and blue edges to $H_{2}$ and a total of at most $6$ vertices in $\mathcal{A}$. For a contradiction, suppose there are at least $11$ vertices with red edges to $H_{2}$ and blue edges to $H_{1}$, say $\mathcal{B}$. By Fact~\ref{Fact:3.4.2}, each part in $\mathcal{C}$ can have blue edges to at most one vertex in $\mathcal{B}$. This means there are at least $7$ vertices and so at least $3$ parts in $\mathcal{B}$ with all red edges to $\mathcal{C}$. Any red edge between these parts would produce a red copy of $B_{5}$ using red edges to $H_{2} \cup \mathcal{C}$ so all edges between these parts must be blue. This produces a blue copy of $B_{5}$ using blue edges to $H_{1}$, meaning that $\mathcal{B}$ contains at most $10$ vertices. In total, we have
$$
|H_{1} \cup H_{2} \cup \mathcal{C} \cup V(\mathcal{A}) \cup V(\mathcal{B})| \leq 4 + 3 + 4 + 6 + 10 = 25
$$
so there are at least $12$ vertices remaining outside. These vertices have all red edges to $H_{1} \cup H_{2}$ and so, by Fact~\ref{Fact:MonoChi}, the proof is complete.

\begin{subcase}
$a=0$.
\end{subcase}

We first claim that $b\geq 5$. If $b\leq 4$ and by Claim~\ref{Clm:5.7}, we have $b+c\leq 12$, and since $3b+2c+d=37$, we get $d\geq 9$. This means $b+c+d\geq 21$, so since $R(B_{5},B_{5})=21$, there must be a monochromatic $B_{5}$, a contradiction.

Thus $b \geq 5$, say with corresponding parts $H_{1},\dots,H_{5}$ each having order $3$. By Claim~\ref{Clm:5.7} and the previous discussion, we must have $c \leq 7$.  Choose one part of order at least $2$, say $H_{6}$. Since $R(K_{3},K_{3})=6$, there must be a monochromatic triangle, say in blue, in the reduced graph induced on $H_{1},\dots,H_{6}$. Say the parts corresponding to this triangle are $H_{1},H_{2},H_{i}$ where $|H_{1}|=|H_{2}|=3$ and $2 \leq |H_{i}| \leq 3$. For the sake of simplicity, set $i = 3$. To avoid creating a blue copy of $B_{5}$, there can be at most $4 - |H_{3}|$ vertices (other than $H_{3}$), say $\mathcal{C}$, with all blue edges to $H_{1} \cup H_{2}$. If there were $5$ parts, say $H_{4}, H_{5}, H_{6}, H_{7}$, and $H_{8}$, with blue edges to $H_{1}$ and red edges to $H_{2}$, then to avoid creating a blue copy of $B_{5}$, at most one of these parts (really at most one vertex) can have blue edges to $H_{3}$, say $H_{4}$. Then all edges from $H_{2} \cup H_{3}$ to $\mathcal{A} = H_{5} \cup \dots \cup H_{8}$ must be red, meaning that there are no red edges within $\mathcal{A}$. With these $4$ parts and all blue edges to $H_{1}$, there is a blue copy of $B_{5}$, a contradiction. This means there can be at most $3$ parts in $\mathcal{A}$. In fact, if any of these parts has order at least $2$, there can be at most $2$ such parts and so $|\mathcal{A}| \leq 6$ and so a total of at most $9$ vertices with blue edges to $H_{1}$ and red edges to $H_{2}$. Similarly there can be at most $9$ vertices with red edges to $H_{1}$ and blue edges to $H_{2}$. This means that
$$
|H_{1} \cup H_{2}| + |H_{3} \cup \mathcal{C}| + |\mathcal{A}| + |\mathcal{B}| \leq 6 + 4 + 9 + 9 = 28.
$$
By Fact~\ref{Fact:MonoChi}, since there are at least $9$ vertices remaining with all red edges to $H_{1} \cup H_{2}$, the proof is complete. 
\end{proof}

Now the main result of this section, the sharp Gallai-Ramsey numbers for several small books.

\begin{theorem}\label{Thm:GR-Bm-Small}
Given integers $k \geq 1$ and $2 \leq m \leq 5$,
$$
gr_{k}(K_{3} : B_{m}) = \begin{cases}
m + 2 & \text{if $k = 1$},\\
(R_{m} - 1) \cdot 5^{(k - 2)/2} + 1 & \text{if $k$ is even,}\\
2 \cdot (R_{m} - 1) \cdot 5^{(k - 3)/2} + 1 & \text{if $k \geq 3$ is odd.}
\end{cases}
$$
\end{theorem}

\begin{proof}
The lower bound follows from Theorem~\ref{Thm:GR-Bm-LowBd}. For the upper bound, let $\ell = \ell_{m}(G)$ be the number of colors that are $m$-inadmissible in a colored complete graph $G$. We will prove the upper bound by induction on $k - \ell$. In fact, we prove the following claim, which immediately implies the theorem.

\begin{claim}\label{Claim:MainSmall}
Given integers $k, \ell, m$ with $k \geq 1$, $2 \leq m \leq 5$, and $0 \leq \ell \leq k$, we have
$$
gr_{k, \ell}(K_{3} : B_{m}) = gr_{k, \ell, m}.
$$
\end{claim}

If $k = 1$, then Claim~\ref{Claim:MainSmall}, and therefore Theorem~\ref{Thm:GR-Bm-Small}, are immediate so next suppose $k = 2$. Since $G$ is a $2$-colored complete graph, we show the following claim.

\begin{claim}\label{Claim:Smallk=2} 
For $2 \leq m \leq 5$, any $2$-coloring $G$ of $K_{n}$ where
$$
n = R_{m} - \ell(m - 1)
$$
with $\ell$ different $m$-inadmissible colors must contain a monochromatic copy of $B_{m}$.
\end{claim}

\begin{proof}
The claim is immediate when $\ell = 0$. When $\ell = k = 2$, there are no admissible colors so each vertex has at most $m - 1$ incident edges in each color. With $n > 2m - 1$ and only two available colors, this is a contradiction, so we may assume $\ell = 1$. Let red be the admissible color and let blue be the inadmissible color.

Choose a pair of vertices in $G$, say $u$ and $v$, with a red edge between them. Then each of $u$ and $v$ has at most $m - 1$ incident blue edges. Let $N_{b}(u)$ (and $N_{b}(v)$) denote the set of vertices with blue edges to $u$ (respectively $v$). Since
$$
|G \setminus \{u, v\}| = R_{m} - (m - 1) - 2 \geq (4m + 1) - (m + 1) \geq 3m,
$$
there must be at least $3m - 2(m - 1) = m + 2$ vertices in $G \setminus (\{u, v\} \cup N_{b}(u) \cup N_{b}(v))$. If we let $W$ be this set of vertices, then $\{u, v\} \cup W$ induces a red copy of $B_{m}$, as claimed.
\end{proof} 

By Claim~\ref{Claim:Smallk=2}, for the remainder of this proof, we may assume $k \geq 3$.

If $k - \ell = 0$, then there are no $m$-admissible colors. With
$$
n = \begin{cases}
17 & \text{ if $m = 2$,}\\
23 & \text{ if $m = 3$,}\\
29 & \text{ if $m = 4$, \text{ and}}\\
33 & \text{ if $m = 5$,}
\end{cases}
$$
by Theorem~\ref{Thm:Stars}, there is a vertex with at least $\frac{2n}{5} > m$ edges in a single color, so this color would be $m$-admissible, a contradiction.

Next suppose $k - \ell = 1$, and consider a G-partition. If there are only two parts, then either both parts have at least $m$ vertices or one part has fewer than $m$ vertices. If both parts have at least $m$ vertices, then neither part contains any edges in this color, reducing $k$ by $1$ and reducing the problem to the case $k - \ell = 0$. On the other hand, if one part has fewer than $m$ vertices, those vertices may be removed, yielding a graph with one more $m$-inadmissible color, again reducing the problem to the case $k - \ell = 0$. More generally, if there are more than two parts, none can have order at least $m$ since the only color with edges to a part of order at least $m$ must be $m$-admissible, yielding a bipartition as above. With all parts having order at most $m - 1$, but only one $m$-admissible color, Lemma~\ref{Lemma:1admissible} implies the existence of a monochromatic $B_{m}$ in the $m$-admissible color, to complete the case $k - \ell = 1$.

This means we may assume $k - \ell \geq 2$. There is a G-partition of $G$, say using colors red and blue. Consider such a G-partition with the smallest number of parts, say $t$. If $t \geq R_{m}$, then by Theorem~\ref{Thm:RamseyBm}, the reduced graph of this partition contains a monochromatic copy of $B_{m}$, a contradiction. We may therefore assume $t \leq R_{m} - 1$. Let $H_{1},H_{2},\dots,H_{t}$ be parts of such a G-partition.

First suppose $t \leq 3$. If $t = 3$, then the reduced graph is a triangle but this contains a bipartition with only one color on the edges between the parts so we may assume $t=2$. Let red be the color of the edges between the two parts. If $|H_{1}| \leq m - 1$, then red is $m$-inadmissible within $H_{2}$ so we remove $H_{1}$ from the graph and apply induction on $k - \ell$ within $H_{2}$. This means that $|H_{1}|, |H_{2}| \geq m$, but then there can be no red edges within either $H_{1}$ or $H_{2}$. This gives
\beqs
n & = & |H_{1}| + |H_{2}|\\
~ & \leq & 2 \left[ gr_{k - 1, \ell, m} - 1\right]\\
~ & < & gr_{k, \ell, m},
\eeqs
a contradiction. We may therefore assume $t \geq 4$ and, by the minimality of $t$, each part of the G-partition has incident edges in both colors that appear in the G-partition. Suppose red and blue are the two colors appearing in the G-partition.

If one of red or blue was $m$-inadmissible in $G$, then since both colors must induce a connected subgraph of the reduced graph, all parts must have order less than $m$, so by Lemma~\ref{Lemma:r=0}, we have $|G| \leq R_{m}'$.
This is a contradiction since $k \geq 3$, so red and blue must be $m$-admissible within $G$.

Let $r$ be the number of parts of the G-partition with order at least $m$, say with
$$
|H_{1}|, |H_{2}|, \dots, |H_{r}| \geq m ~ ~ ~ ~ \text{and} ~ ~ ~ ~ |H_{r + 1}|, |H_{r + 2}|, \dots, |H_{t}| \leq m - 1.
$$

If $r \geq 4$ and $t \geq 6$, then any choice of $6$ parts containing the $4$ parts $\mathscr{H} = \{H_{1}, H_{2}, H_{3}, H_{4}\}$ will contain a monochromatic triangle in the reduced graph. Such a triangle must contain at least one part from $\mathscr{H}$, meaning that the corresponding subgraph of $G$ must contain a monochromatic copy of $B_{m}$, a contradiction. Thus, we may assume either $4 \leq t \leq 5$ or $r \leq 3$.  The case $r = 0$ follows from Lemma~\ref{Lemma:Smallr=0}, and the remaining cases, for $r = 1$, $r = 2$, $r = 3$, and $r \geq 4$ follow from exactly the same argument as in the proof of Theorem~\ref{Thm:MainBm}. This completes the proof of this case and the proof of Claim~\ref{Claim:MainSmall} and therefore Theorem~\ref{Thm:GR-Bm-Small}.
\end{proof}

\section{Conclusion}\label{Sec:Con}

Putting Theorem~\ref{Thm:GR-Bm-LowBd} and Theorem~\ref{Thm:MainBm} together, we have obtained the following general bounds on the Gallai-Ramsey numbers for books.

\beqs
&&\begin{cases}
(4 + o(1))m \cdot 5^{(k - 2)/2} + 1 & \text{if $k$ is even,}\\
2 \cdot (4 + o(1))m \cdot 5^{(k - 3)/2} + 1 & \text{if $k$ is odd}
\end{cases}\\
& \leq & gr_{k}(K_{3} : B_{m}) \\
& \leq &\begin{cases}
m + 2 - \ell & \text{ if $k = 1$,}\\
(4 + o(1))m\ln [(4 + o(1))m] \cdot 5^{\frac{k-2}{2}}+1 - (m - 1)\ell & \text{ if $k$ is even,}\\
2 \cdot (4 + o(1))m\ln [(4 + o(1))m] \cdot 5^{\frac{k-3}{2}}+1 - (m - 1)\ell & \text{ if $k \geq 3$ is odd.}
\end{cases}
\eeqs

Naturally it would be very nice to sharpen these bounds. As in the small cases presented in Section~\ref{Sec:Small}, we believe the lower bound to be sharp and so we offer the following conjecture.

\begin{conjecture}
If $B_{m}$ is the book with $m$ pages, $B_{m} = K_{2} + \overline{K_{m}}$, then for $k \geq 2$,
$$
gr_{k}(K_{3} : B_{m}) = \begin{cases}
(R(B_{m}, B_{m}) - 1) \cdot 5^{(k - 2)/2} + 1 & \text{if $k$ is even,}\\
2 \cdot (R(B_{m}, B_{m}) - 1) \cdot 5^{(k - 3)/2} + 1 & \text{if $k$ is odd.}
\end{cases}
$$
\end{conjecture}


\end{document}